\documentclass[a4paper,12pt,draft]{amsart}

\usepackage[cp1252]{inputenc}
\usepackage{stmaryrd}
\usepackage{amsthm}
\usepackage{amscd}
\usepackage{amssymb}
\usepackage[all,cmtip]{xy}
\usepackage{array}
\usepackage{enumerate}

\begin{document}



\newtheorem{THM}{{\!}} [section]
\newtheorem{THMX}{{\!}}
\renewcommand{\theTHMX}{}
\newtheorem{thm}{Theorem} [section]
\newtheorem{cor}[thm]{Corollary}
\newtheorem{lem}[thm]{Lemma}
\newtheorem{clm}[thm]{Claim}
\newtheorem{prop}[thm]{Proposition}
\newtheorem{question}[thm]{Question}
\newtheorem{fact}[thm]{Fact}

\theoremstyle{definition}
\newtheorem{defi}[thm]{Definition}
\newtheorem{axiom}[thm]{Axiom}
\newtheorem{rem}[thm]{Remark}
\newtheorem{souv}[thm]{Reminder}
\newtheorem{ex}[thm]{Example}
\newtheorem{exo}[thm]{Exercice}
\newtheorem{dem}[thm]{Proof}

\newcommand{\Z}{\mathbb{Z}}
\newcommand{\N}{\mathbb{N}}
\newcommand{\Q}{\mathbb{Q}}
\newcommand{\R}{\mathbb{R}}
\newcommand{\C}{\mathbb{C}}
\newcommand{\pgcd}{\hbox{{\rm pgcd}}}
\newcommand{\CM}{\mathcal M}
\newcommand{\sub}{\subseteq}
\newcommand{\CN}{\mathcal N}
\newcommand{\ra}{\rangle}
\newcommand{\la}{\langle}
\newcommand{\CF}{\mathcal F}
\newcommand{\Aut}{\mathrm{Aut}}
\newcommand{\dom}{\mbox{dom}}
\newcommand{\noi}{\noindent}
\newcommand{\acl}{\mathrm{acl\,}}
\newcommand{\acle}{\mathrm{acl\,(\emptyset)}}
\newcommand{\dcl}{\mathrm{dcl\,}}
\newcommand{\rk}{\mathrm{rk}}
\newcommand{\df}{\mathrm{Def}}
\def\CL{\mathcal{L}}
\def\CX{\mathcal{X}}
\def\CY{\mathcal{Y}}

\title[Curves in geometric structures]{Some (non-)elimination results\\ for curves in geometric structures}

\author{Serge Randriambololona}
\address{S. Randriambololona, Department of Mathematics\\  The University of Western Ontario\\ London, Ontario N6A 5B7 Canada}
\email{srandria@uwo.ca}

\author{Sergei Starchenko}
\address{S. Starchenko, Department of Mathematics\\ University of Notre Dame\\ Notre Dame IN 46556  USA}
\email{sstarche@nd.edu}

\thanks{The second author was partially supported by NSF grant DMS-0701364}

\begin{abstract}
  We show that the first order structure whose underlying universe is
  $\C$ and whose basic relations are all algebraic subset of $\C^2$
  does not have quantifier elimination. Since an algebraic subset of
  $\C ^2$ needs either to be of dimension $\leq 1$ or to have a
  complement of dimension $\leq 1$, one can restate the former result
  as a failure of quantifier elimination for planar complex algebraic
  curves.  We then prove that removing the planarity hypothesis
  suffices to recover quantifier elimination: the structure with the universe
   $\C$ and a predicate for each algebraic subset of
  $\C^n$ of dimension $\leq 1$ has quantifier elimination.

\end{abstract}

\maketitle

\section{Introduction}
\label{intro}

The theory of structure generated by binary relations definable in an o-minimal structure was studied in \cite{MeRuSt}.
In particular, Theorem 3.2.\cite{MeRuSt}  implies the following proposition:

\begin{prop}
\label{omin}
Let $\CM $ be an o-minimal structure with universe $M$, and let 
$\mathcal B(\CM)$ be the first-order structure whose underlying set is $M$ and whose basic relations are all subsets of $M^2$ which are 
$\emptyset$-definable in $\CM$.
The theory of $\mathcal B(\CM)$ has quantifier elimination. 
\end{prop}

As an immediate consequence of quantifier elimination, the structure $\mathcal B(\CM)$ has trivial geometry (Lemma 1.8 in \cite{MeRuSt}).

Also, partially motivated by a restricted version of Zil'bers' Conjecture, various 
reducts  of the field of complex numbers have been investigated (see for example \cite{Lv,MaPi,Rb}), 
and  it is
natural to ask whether a complex analogue of the previous proposition holds: does 
the structure $\mathcal B(\overline \C_{\C})$, obtained by equipping
the universe $\C$ with a predicate for each complex algebraic
constructible subset of $\C ^2$, also eliminates quantifiers?  Note
that the arity two is the only arity where this question occurs: for
arity three and above, we recover the full structure of field on $\C$
and thus get quantifier elimination; as for arity zero and one,
elimination of quantifiers is clear.

Section \ref{sectionexemple} will answer negatively this question: it
provides a counter-example to the elimination of quantifiers for
$\mathcal B(\overline \C_{\C})$. Still, one may ask what sets are
definable in $\mathcal B(\overline \C_{\C})$, and what is its (combinatorial)  geometry. These questions are
 answered in Section \ref{sectionmain}: let $\mathcal
C(\overline \C_{\C})$ denote the first-order structure whose
underlying universe is $\C$ and whose basic relations are all the
subsets of cartesian products of $\C$, definable in the field structure, of
dimension $\leq 1$ (the constructible curves). We first note that
$\mathcal B(\overline \C_{\C})$ is a reduct (in the sense of
definability) of $\mathcal C(\overline \C_{\C})$. We then show that
$\mathcal C(\overline \C_{\C})$ eliminates its quantifiers, and deduce that it has a trivial geometry.
In particular  $\mathcal B(\overline \C_{\C})$ is a proper reduct
(in the sense of definability) of the field of complex
number.  The latter results are proven in the more general setting of {\it
  geometric structures}. In Section \ref{parameters} we discuss the
role of algebraic closure versus definable closure in this quantifier
elimination. In Section \ref{high}, we generalize the
construction of Example \ref{exemple} to higher arity: for any fixed
natural number $n$ we consider the structure $\mathcal C_n(\overline
\C_{\C})$ on $\C$ whose basic relations are the subsets of $\C ^n$,
definable in the field of complex numbers, of dimension $\leq 1$. We show that none
of the $\mathcal C_n(\overline \C_{\C})$ has quantifier elimination
(Example \ref{exemple} showing this fact for $n=2$). Finally in Section \ref{quantifier}, 
we discuss which sets are definable in those structures (allowing quantifications).

\medskip
{\bf Remark.} After the paper had been submitted we discovered that a result similar to our Theorem 3.5 
was also proved recently by M.~C.~Laskowski in \cite{Lask}. 
\medskip

{\bf  Acknowledgment.}
We thank the referee for the careful reading of the paper, providing constructive comments, and help
in improving the content of this paper.

\section{Non-elimination for binary relations}
\label{sectionexemple}

Recall from the Introduction that $\mathcal B(\overline \C_{\C})$
denotes the first-order structure whose universe is $\C$ and whose
basic relations are all subsets of $\C ^2$ which are $\C$-definable in
$\overline \C$, the field of complex numbers.  We show that $\mathcal
B(\overline \C_{\C})$ does not eliminate its quantifier:

\begin{ex}
\label{exemple}
Let $R$ be the binary definable relation
$$R(y,s) \Longleftrightarrow \exists z\bigl( z\neq y \wedge y^4+y=z^4+z \wedge s=y+z\bigr)$$
and consider the ternary relation
\[
T(s_1,s_2,s_3)\Longleftrightarrow \exists y\bigl( R(y,s_1)\wedge R(y,s_2) \wedge R(y,s_3)\bigr).
\]

The subset of $\C^3$ defined by the relation $T$ is definable in $\mathcal B(\overline \C_{\C})$ but is {\bf not quantifier-free definable} in 
$\mathcal B(\overline \C_{\C})$.

\end{ex}

\begin{proof}
  Let $\mathcal M$ be a proper elementary extension of
  $\overline \C $ and $M$ its universe.  We fix (any) $a\in M\setminus \C$, and let $\{ \zeta_1,
  \zeta_2, \zeta_3, \zeta_4 \}$ be the four roots of the polynomial
  $X^4+X+a$ in $M$. Note that $a$ is transcendental over $\C$.

  We first claim that $\mathrm{Aut}(\mathcal M \, / \, \C)$ (the group
  of automorphisms of $\CM$ fixing $\C$) acts
  totally transitively on the set $\{ \zeta_1, \zeta_2, \zeta_3,
  \zeta_4 \}$ (which really has four distinct elements). Since $\CM$
  is an algebraically closed field, every element of
  $\mathrm{Gal}\big( X^4+X+a\, / \,\C (a)\big)$ extends to an
  automorphism of $\mathcal M$, hence it is sufficient to show that
  $\mathrm{Gal}\big( X^4+X+a\, / \,\C (a)\big)$ is the symmetric group 
  $\mathfrak S _4$.

  Galois theory (see for instance Theorem 13.4 in \cite{Mo}) tells us that this happens if and only if both the polynomial $X^4+X+a$ and
  its resolvent $X^3-4aX-1$ are irreducible over $\C (a)$ and the discriminant $256a^3-27$ is not a square in $\C(a)$.

We show that $X^3-4aX-1$ is irreducible over $\C(a)$. Assume
not, then $X^3-4aX-1$ has a root in $\C(a)$. Let $\alpha\in \C(a)$ be
such a root, and let $p(X),q(X)\in \C[X]$ be relatively prime polynomials such
that $\alpha=p(a)/q(a)$.  We have $p^3(a)-4ap(a)q^2(a)-q^3(a)=0$, and since $a$ is transcendent over $\C$, 
the equality $p^3(X)-4Xp(X)q^2(X)-q^3(X)=0$ holds in $\C[X]$. If $\gamma\in
\C$ is a root of $q(X)$ then it follows from above equation that
$p(\gamma)=0$. Since $p,q$ are relatively prime, they have no common
roots, hence $q$ must be a constant polynomial. But then $a$ would be algebraic over $\C$, a
contradiction. 

Using the same arguments it is not hard to see that $X^ 2-256a^3-27$ is irreducible and 
that $X^4+X+a$ has no root in $\C(a)$. 

To show that $X^4+X+a$ is irreducible over $\C (a)$, it remains to prove that it can not be written as a product of two quadratic
polynomials. Assume $X^4+X+a=(X^2+\alpha_1X+\beta_1)(X^2+\alpha_2X+\beta_2)$ with
$\alpha_i,\beta_i\in \C(a)$.  Expanding the right side we obtain the equations
\begin{center}
\begin{tabular}{lclc}
(i)& $\alpha_1+\alpha_2=0$& (ii)& $\beta_2+\alpha_1\alpha_2+\beta_1=0$\\ 
(iii)& $\alpha_1\beta_2+\alpha_2\beta_1=1$& (iv)& $\beta_1\beta_2=a$.
\end{tabular}
\end{center}

Combining (i) with (ii) and (iii), we get 
\[
\beta_2+\beta_1=-\alpha_1^2, \ \beta_2-\beta_1=1/\alpha_1,
\]
and therefore 
\[
4\beta_1\beta_2= (\beta_2+\beta_1)^2-(\beta_2-\beta_1)^2=\alpha_1^4-1/\alpha_1^2.
\] 
By (iv), we have 
\[
4a=\alpha_1^4-1/\alpha_1^2.
\]
If for a contradiction $\alpha_1$ would belong to $\C(a)$ then $t=\alpha_1^2$ would belong to $\C (a)$.
But by the previous equation we have $t^3-4at-1=0$: this would contradict the irreducibility over $\C (a)$ of $X^3-4aX-1$ proved earlier.
 
Thus the groups $\mathrm{Gal}\big( X^4+X+a\, / \,\C (a)\big)$ is $\mathfrak S _4$.

\medskip

Consider the two triplets $(s_1,s_2,s_3)$ and $(s'_1,s'_2,s'_3)$ of elements
of $M$ defined by $s_i=\zeta _i+\zeta_4$ for $i=1\ldots 3$ and
$s'_{\tau(1)}=\zeta_{\tau (2)}+\zeta_{\tau (3)}$ for all $\tau \in
\mathfrak S _3$ (see figure below).

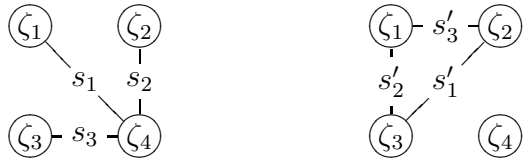
\begin{center}
\begin{tabular}{lcr}
\entrymodifiers={+[o][F]} 
\xymatrix{ \zeta _1 \ar@{-}[rd]|*+[o]{s_1} & \zeta_2 \ar@{-}[d]|*+{s_2} \\ \zeta_3 \ar@{-}[r]|*+{s_3} & \zeta _4} 
& \hspace{2cm} 
& \entrymodifiers={+[o][F]}
\xymatrix{\zeta _1 \ar@{-}[d]|*+[o]{s'_2} \ar@{-}[r]|*+[o]{s'_3} & \zeta_2 \ar@{-}[ld]|*+[o]{s'_1} \\ \zeta _3 & \zeta _4}
\end{tabular}
\end{center}

Any two triplets of distinct $\zeta _m$'s are $\Aut (\mathcal M
/\C)$-conjugate, and we get the elementary equivalence $(s_i, s_j)
\equiv _{\C} (s'_k, s'_l)$ (in the sense of $\overline \C$)  for any $1\leq i \neq j \leq 3$ and any
$1\leq k \neq l \leq 3$.

In particular the elementary equivalence for $(i,j)=(k,l)$ insures that if $T$ were to be quantifier-free definable in $\mathcal B (\overline \C _{\C})$ then we 
would have $ T(s_1,s_2,s_3)$ if and only if $ T(s'_1,s'_2,s'_3)$: $T$ would be equivalent to a boolean combination of formul\ae\ in 
$\overline \C _{\C}$, each of which involving only two of the three possible variables (say indexed by $(i,j)$); such a formula would be satisfied by the 
corresponding subtuple of $(s_i,s_j)$ if and only if were satisfied by the subtuple $(s'_i,s'_j)$.

But $(s_1,s_2,s_3)$ does satisfy $T$ whereas we will show that $(s'_1,s'_2,s'_3)$ does not. 
Suppose for a contradiction that $T(s'_1,s'_2,s'_3)$ holds; then there are  $\{ \zeta ' _1,  \zeta ' _2, \zeta ' _3, \zeta ' _4 \}$ such that 
${\zeta '_i}  ^4 + \zeta' _i ={\zeta '_j} ^4 + \zeta' _j$ and 
$s' _i=\zeta ' _i+\zeta '_4$. Thus 
\[
-2\zeta_4= 2(\zeta  _1+  \zeta  _2+\zeta  _3)=s'_1+s'_2+s' _3=  \zeta ' _1+  \zeta ' _2+\zeta ' _3+3\zeta ' _4=2\zeta ' _4
\] and 
\[
\zeta '_i = s'_i-\zeta'_4=s'_i +\zeta _4=\zeta_1+\zeta_2+\zeta_3-\zeta_i+\zeta_4=-\zeta _i.
\] Therefore 
\[
-a-2\zeta _i= \zeta _i  ^4 - \zeta _i = {\zeta' _i}  ^4 + \zeta' _i= {\zeta' _j}  ^4 + \zeta' _j=\zeta _j ^4 - \zeta _j=-a-2\zeta _j 
\] 
for all $i\neq j$: this contradicts the 
fact that the $\zeta_i$'s are distinct.  
\end{proof}

We get a slightly stronger result than announced: let 
$\mathcal B(\overline \C_{\emptyset})$ denote the first-order structure whose universe is $\C$ and 
whose basic relations are all subsets of $\C ^2$ which are {\bf $\emptyset$-definable} in $\overline \C$. Then 
$\mathcal B(\overline \C_{\emptyset})$ defines subsets of $\C ^3$ which are not quantifier-free definable in 
$\mathcal B(\overline \C_{\C})$.

\section{Elimination for curves}
\label{sectionmain}

We have seen in the previous section that existential quantifiers can be used to bind variables together and define 
{\it essentially} non-binary algebraic relations from binary ones.  Still one can ask how complicated can be a set defined using only binary relations. 
Could it be, for instance, that $\mathcal B (\overline \C _{\C})$ and the full field structure  $\overline \C$ are interdefinable ? 
We will show that it is not the case.

First note that each subset of $\C^2$ definable in $\overline \C$ (with parameters) is a boolean combination of subset of $\C ^2$ definable in 
$\overline \C$ (with parameters) of {\bf dimension smaller or equal to $1$} and vice versa (where ``{\it dimension}'' refers to the $\acl$-dimension in the sense 
of $\overline \C $):  

\begin{fact}
\label{twoiscurve}
Let $X\subseteq \C ^2$ be definable in $\overline \C$, with parameters. Then either $\dim X \leq 1$ or $\dim (\C ^2 \setminus X) \leq 1$.
\end{fact}
 
By Fact \ref{twoiscurve}, we can view Example \ref{exemple} as showing that the theory of $\C$ equipped with a predicate for each planar algebraic curve 
does not have quantifier elimination.
 
However we will show that if we remove the requirement that the curves are \emph{planar}, the quantifier elimination holds. 
As a consequence $\mathcal B (\overline \C _{\C})$ will be shown to have trivial geometry and thus to be a proper reduct of $\overline{\C}$.

The results from this section not only hold in $\overline \C$ but also in the more general setting of {\it geometric structures}. 
O-minimal structures, strongly minimal structures (such as algebraic closed fields), $p$-adic fields or algebraically closed valued fields 
with a predicate for their valuation ring all are geometric structures. 

\begin{defi}
Recall that the structure $\mathcal M$ is said to be a {\it geometric structure} if it satisfies 
\begin{enumerate}
\item the Exchange Principle: $a\in \acl (bC)\setminus \acl (C) \Rightarrow b\in \acl (aC)$
\item Uniform Finiteness Property: given a formula $\psi$, there is an integer $k$ such that for each tuple $\mathbf{a}$ 
the set $\{b |\mathcal M \models\psi (b,\mathbf{a})\}$ is either infinite or of side $\leq k$.
\end{enumerate}
Note that this property is a property of the theory of $\mathcal M$.
\end{defi}

In the sequel we will work in a fixed geometric structure $\CM$ and call dimension the $\acl$-dimension 
for its definable sets: if $\Phi$ is a formula (with parameters $B$) defining such a set $X$ and $\widetilde{\CM}$ is a saturated extension of $\CM$, 
the dimension of $X$ is the maximal $d$ for which there exists $(a_1,\ldots,a_n)$ satisfying $\Phi$ (in $\widetilde{\CM}$) and a subtuple 
$(a_{i_1},\ldots, a_{i_d})$ of $(a_1,\ldots,a_n)$ of length $d$ such that $a_{i_{j+1}}\notin \acl (\{a_{i_1},\ldots,a_{i_j}\}\cup B)$. 
(This quantity is independent of the choice of the formula $\Phi$, the parameters $B$ and the structure $\widetilde{\CM}$.)

\begin{defi}
Let $\CM$ be a geometric structure with universe $M$.

A set $X\subseteq M^n$ definable with parameters from $A\subseteq M$, of dimension $\leq 1$ will be called an ($A$-definable) {\it $n$-curve}. 

An ($A$-definable) {\it curve} is an ($A$-definable) $n$-curve for some $n$.

A set $\tilde C \subseteq M^m$ is said to be an ($A$-definable) {\it
  cylinder based on a $n$-curve} if there are some indices $1\leq i_1
< \ldots < i_n\leq m$ and an $A$-definable $n$-curve $C$ such that
$\tilde C=\{(x_1,\ldots,x_m)\in M^m | (x_{i_1},\ldots,x_{i_n})\in C
\}$.

A set is called an ($A$-definable) {\it curve-based cylinder} if it is
an ($A$-definable) cylinder whose base is a $n$-curve for some $n\in
\N$.
\end{defi}

\begin{rem}
These definitions reflect the fact that one needs to pay attention to the variables used: the formula $x_1=x_2$ viewed as a formula in the variables $x_1$ and 
$x_2$ defines a $2$-curve but if we add a dummy variable $x_3$ it defines a cylinder based on a $2$-curve, of dimension $2$.

Dummy variables and cylinders allow to think about ``boolean combinations of curves" involving different sets of variables, as shown in Example \ref{exemple}.
\end{rem}

As previously announced, the aim of the section is to prove Theorem \ref{main} which easily implies that the structure obtained 
by equipping $\C$ with predicate for each algebraic subset of $\C ^2$ is a proper reduct of the field structure.

\begin{thm}
\label{main}
Let $\CM$ be a geometric structure with universe $M$. 
The structure $\mathcal C (\CM_{\acl(\emptyset)})$ obtained by equipping $M$ with predicates for each $\acl(\emptyset)$-definable curve has quantifier 
elimination. 
\end{thm}

\begin{proof}
Without loss of generality, we can assume that $\mathcal M$ is sufficiently saturated.

By syntactic arguments, we only need to consider formul\ae \ of the form 
\begin{equation}
\label{ori}
\exists y \bigwedge_{i=1} ^r C_i(\mathbf{x}^i,y) \wedge  \bigwedge_{j=r+1} ^{r+s}  \neg C_j(\mathbf{x}^j,y)
\end{equation}
\begin{itemize}
\item where each $\mathbf{x}^k$ denotes $l_k$-tuples of variables among $(x_1,\ldots , x_n)$
\item and each $C_k$ denotes a formula in the $(l_k+1)$-subtuple $\mathbf{x}^ky$ of free variables among those of the tuple $(x_1,\ldots , x_n,y)$, 
defining a $(l_k+1)$-curve 
\end{itemize}
and prove that they define a boolean combination of $\acl(\emptyset)$-definable curve-based cylinders.  
\smallskip

In what follows $C(\mathbf{w},y)$, $C_i(\mathbf{w},y)$ and $E(\mathbf{w},y)$ will denote formul\ae \ with distinguished last variable $y$, 
defining $\acl(\emptyset)$-definable $(|\mathbf{w}|+1)$-curves. 
Similarly $\phi_l (\mathbf{x})$ will denote a formula defining a boolean combination of $\acl(\emptyset)$-definable curve-based cylinders. 
(Note that the bases of each of these cylinders may involve different tuples of coordinates among those of $\mathbf x$.)

\begin{lem}
\label{grouping}
Any formula $\chi(\mathbf{x},y)$ of the form
\begin{equation} \label{beforegrouping}
\bigwedge_{i=1} ^r C_i (\mathbf{x}^i,y) 
\end{equation}
is equivalent to a disjunction
\begin{equation} \label{rat} 
E(\mathbf{x}',y) \vee \bigvee _{l=1} ^L \big( y=q_l \wedge \phi_l (\mathbf{x}') \big)
\end{equation}
where $\mathbf{x}'$ is the sub-tuple of $\mathbf x$ of all those variables involved in some of the tuples $\mathbf{x}^i$ ($i=1,\ldots,r$), 
$q_1,\dotsc,q_L$'s are  elements of $\acl (\emptyset)$, and 
$\mathcal M\models E(\mathbf{x}', y) \rightarrow \bigwedge_{l=1} ^L y\neq q_l $.
\end{lem}

\begin{proof}  
Consider $\boldsymbol{\xi}\gamma$ in $M^{|\mathbf{x}|+1}$ such that
 $\chi(\boldsymbol{\xi},\gamma)$ holds and let 
the sub-tuples $\boldsymbol{\xi} ^i$ and $\boldsymbol{\xi}'$ of $\boldsymbol{\xi}$ correspond, respectively, 
to the sub-tuples $\mathbf{x} ^i$ and $\mathbf{x}'$ of $\mathbf{x}$.

By the Exchange Property, either $\gamma$ belong to 
$\acl(\emptyset)$ or each coordinate of $\boldsymbol{\xi} '$ belongs to $\acl (\gamma)$. 
In the latter case, $\boldsymbol{\xi}' \gamma$ satisfies some formula 
defining an $\acl(\emptyset)$-definable curve.


Since $\CM$ is saturated enough, we obtain by compactness that for some $\acl(\emptyset)$-definable curve $C(\mathbf{x'},y)$ and 
$q_1,\dotsc,q_L\in \acl(\emptyset)$ such that
\[ \CM\models \chi(\mathbf{x},y)\rightarrow 
\bigl( \bigvee_{i=1}^Ly=q_i \, \vee  C(\mathbf{x'},y)\bigr). \]
We can take 
\[  C(\mathbf{x'},y) \wedge \bigwedge_{i=1} ^r C_i (\mathbf{x}^i,y) 
\wedge  \bigwedge_{i=1} ^L y\not = q_i \mbox{ for } E(\mathbf{x}',y)\]
and 
\[  \bigwedge_{i=1} ^r C_i (\mathbf{x}^i,q_l) \mbox{ for } \phi _l(\mathbf{x}'),\ 
l=1,\dotsc,L.\] 
\end{proof}

\begin{lem}
\label{positive}
Let $d$ be a natural number. Any formula of the form 
\begin{equation} \label{apositive}
\exists ^{\geq d} y \bigwedge_{i=1} ^r C_i (\mathbf{x}^i,y) 
\end{equation}
defines a boolean combination of $\acl (\emptyset)$-definable curve-based cylinder.
\end{lem}

\begin{proof}
By Lemma \ref{grouping}, the formula (\ref{apositive}) is equivalent to some 
\[
\exists ^{\geq d} y \,  E(\mathbf{x}',y) \vee \bigvee _{l=1} ^L \big( y=q_l \wedge \phi_l (\mathbf{x}') \big)
\] with 
$E(\mathbf{x}', y) \rightarrow \bigwedge_{l=1} ^L y\neq q_l$ and the 
$q_l$'s all distinct. It is thus  
also equivalent to the disjunctions of the formul\ae
\[
\big( \exists ^{\geq \, d -|\widetilde L|} \, y \,  E(\mathbf{x}',y)\big) \wedge \big( \bigwedge_{l\in \widetilde L} \phi_l(\mathbf{x}') \wedge 
\bigwedge_{l\notin \widetilde L} \neg \phi_l(\mathbf{x}') \big)
\] 
as $\widetilde L$ ranges among the subsets of $\{1,\ldots ,L \}$.

Since for every $e\in \mathbb N$  the set   
$\{ \mathbf{x}' | \exists^{\geq e} y E(\mathbf{x}',y)\}$ has dimension $\leq 1$,
 any formula of the form $\exists^{\geq e} y E(\mathbf{x}',y)$ defines a
curve.

\end{proof}

We proceed by induction on $s$ (the number of negations involved) to
show that that any formula of the form
\begin{equation*}
 \exists y \bigwedge_{i=1} ^r C_i(\mathbf{x}^i,y) \wedge  \bigwedge_{j=r+1} ^{r+s}  \neg C_j(\mathbf{x}^j,y)
\tag{1}
\end{equation*}
defines a boolean combination of $\acl (\emptyset)$-definable curve-based cylinders.

The result is proved for $s=0$ by Lemma \ref{positive}. Fix $s\geq 1$, a formula of the form (\ref{ori}) and suppose that the induction hypothesis  holds for any 
$s'<s$.

Note first that (\ref{ori}) is equivalent to the formula  
\[
\exists y \bigwedge_{i=1} ^r C_i (\mathbf{x}^i,y) \wedge  \bigwedge_{j=r+1} ^{r+s}  \neg \big(C_j (\mathbf{x}^j,y) \wedge \bigwedge_{i=1} ^r C_i (\mathbf{x}^i,y) \big).
\]
This formula says that there is a $y$ satisfying $\bigwedge_{i=1} ^r C_i (\mathbf{x}^i,y)$ but not satisfying  
$\bigvee_{j=r+1} ^{r+s} \big( C_j (\mathbf{x}^j,y) \wedge  \bigwedge_{i=1} ^r C_i (\mathbf{x}^i,y)\big)$.

Thus (\ref{ori}) is equivalent to the disjunction of (\ref{finite}) and (\ref{infinitedist}) below: 
\begin{equation} \label{finite}
\bigvee _{d\in \N}  \bigg( \Big( \exists^{\geq d+1} y\  \bigwedge_{i=1} ^r C_i (\mathbf{x}^i,y) \Big) \wedge
\Big( \exists^{=d} y\ \bigvee_{j=r+1} ^{r+s} \big( C_j (\mathbf{x}^j,y) \wedge  \bigwedge_{i=1} ^r C_i (\mathbf{x}^i,y)\big) \Big) \bigg) 
\end{equation} 
(the case when there are only finitely many $y$'s satisfying the condition
$\bigvee_{j=r+1} ^{r+s} \big( C_j (\mathbf{x}^j,y) \wedge  \bigwedge_{j=i} ^r C_i (\mathbf{x}^i,y)\big)$) 

\vspace{1mm}
\noi and the formula 
\begin{multline} \label{infinitedist}
\bigvee_{j=r+1} ^{r+s} \bigg( \Big( \exists^{\infty} y\  C_j (\mathbf{x}^j,y) \wedge  \bigwedge_{i=1} ^r C_i (\mathbf{x}^i,y)\Big) \wedge\\ 
\Big( \exists y \bigwedge_{i=1} ^r C_i (\mathbf{x}^i,y) \wedge  \bigwedge_{k=r+1} ^{r+s}  \neg C_k (\mathbf{x}^k,y) \Big) \bigg) 
\end{multline}
(the case where there is some $j>r$ for which there are infinitely many $y$'s satisfying the $C_j (\mathbf{x}^j,y) \wedge  
\bigwedge_{i=1} ^r C_i (\mathbf{x}^i,y)$). 

Since the formula $C_j (\mathbf{x}^j,y)$  defines a subset of $M^{|\mathbf{x}^j|+1}$ of dimension $\leq 1$, we get that 
\[
\exists^{\infty } y\   C_j (\mathbf{x}^j,y) \leftrightarrow \bigvee_{p=1} ^{P_j} \mathbf{x}^j = \mathbf{r}_p^j
\] 
for some finite collection of tuples $\mathbf{r}_p ^j$, of elements of $\acl (\emptyset)$.
Thus (\ref{infinitedist}) is equivalent to a disjunction of formul\ae \ of the form
\begin{equation}\label{lessvariables}
(\mathbf{x}^{j_0} = \mathbf{r}_p ^{j_0} )  \wedge 
\exists y  \, \bigwedge_{i=1} ^r C_i (\mathbf{x}^i,y) \wedge \neg C_j (\mathbf{r}_p ^{j_0},y) 
\wedge \bigwedge_{\substack{j=r+1\\j\neq j_0}} ^{r+s}  \neg C_j (\mathbf{x} ^j,y) 
 .
\end{equation} 

Observe here that the formula $\neg C_{j_0} (\mathbf{r}_p ^{j_0},y)$ in the free variables $\mathbf xy$ is not only the \emph{negation of} a 
formula defining a curve-based cylinder with parameters from $\acl (\emptyset)$: the formula $\neg C_{j_0} (\mathbf{r}_p ^{j_0},y)$ in the free 
variables $\mathbf xy$ 
defines also {\it itself} a curved-based cylinder definable over $\acl (\emptyset)$ (for it only concerns the distinguished variable $y$). 
Therefore the induction hypothesis implies that each formula of the form (\ref{lessvariables}) defines a boolean combination of 
$\acl (\emptyset)$-definable curve-based cylinders.   
\smallskip

It only remains to prove that the formula (\ref{finite}) is equivalent to a boolean combination of curve-based cylinders. 
By the Uniform Finiteness Property we can replace the infinite disjunction in (\ref{finite}) by a finite one.

Applying the Inclusion-Exclusion Formula for finite sets $$|\bigcup _{h=1} ^t X_h|=\sum_{H\subseteq \{1 , \ldots t \}} (-1)^{|H|+1} \cdot |\bigcap _{h\in H} X_h|$$ 
we get that   \[ 
\exists^{=d} y\ \bigvee_{j=r+1} ^{r+s} \big( C_j (\mathbf{x}^j,y) \wedge  \bigwedge_{i=1} ^r C_i (\mathbf{x}^i,y) \big)
\]
is equivalent to a boolean combination of formul\ae \ of the form 
\[
\exists^{=e} y\ \bigwedge_{j\in J} C_j (\mathbf{x}^j,y)  \wedge  \bigwedge_{i=1} ^r C_i (\mathbf{x}^i,y)
\]
 for some $J$'s 
ranging among subsets of $\{r+1,\ldots ,r+s\}$ and some natural number $e$ not larger than $d$.
Using Lemma \ref{positive}, each of these latter formul\ae \ is equivalent to a boolean combination of one-dimensional formul\ae \ with parameters in 
$\acl (\emptyset)$.
\end{proof}

Quantifier elimination for $\mathcal C (\overline \C_{\C})$ (the structure obtained by equipping $\C$ with a predicate for each curve 
$\C$-definable in $\overline \C$) easily implies that:

\begin{cor}
The structure $\mathcal C(\overline \C_{\C})$ has trivial geometry (that is $\acl(A)=\bigcup_{a\in A} \acl(\{a \})$ for all $A\subseteq \C$.
\end{cor}

In particular we get that:

\begin{cor}
The structure $\mathcal C(\overline \C_{\C})$ is a proper reduct of $\overline{\C}$.
\end{cor}

Since Lemma \ref{twoiscurve} insures that $\mathcal B(\overline \C_{\C})$ (the structure obtained by equipping $\C$ with a predicate for each subset 
of $\C^2$ which is $\C$-definable in $\overline \C$; see Section \ref{sectionexemple}) is a reduct (in the sense of definability) of $\mathcal C(\overline \C_{\C})$, 
the structure $\mathcal B(\overline \C_{\C})$ is a proper reduct (in the sense of definability) of the structure $\overline \C$ and has trivial geometry.

\section{Algebraic and definable closure}
\label{parameters}

In the construction of Example \ref{exemple}, a key fact is that one
can not distinguish the four roots of the polynomial $X^4+X+a$, which
is an illustration that algebraic closure and definable closure are
two different notions in $\overline \C$. Proposition \ref{acl} below insures that
this condition is needed: if we consider a geometric structure $\CM$
on the universe $M$ for which $\acl()=\dcl()$ then the
structure $\mathcal C_2(\CM_{\acl(\emptyset)})$ (whose universe is $M$
and basic relations are all the $\acl(\emptyset)$-definable subsets of
$M^2$ of dimension $\leq 1$) eliminates its quantifiers.

\begin{prop}
\label{acl}
Consider a geometric structure $\CM$ on the universe $M$ such that for
all $A\subseteq M $ we have $\acl(A)= \dcl(A)$.

\begin{enumerate}
\item Any subset of $M^n$ of dimension $\leq1$ definable in $\CM$ over $\acl(\emptyset)$ is a boolean combination of cylinders, each of whose basis is either the 
graph of a function of one variable $\emptyset$-definable in $\CM$ or an element of $\dcl(\emptyset)$.
\item In particular the structure $\mathcal C_2(\CM_{\emptyset})$ on $M$ generated by all $\emptyset$-definable subsets of $M^2$ of dimension $\leq 1$ 
and the structure $\mathcal C(\CM_{\acl(\emptyset)})$ on $M$ generated by all $\acl(\emptyset)$-definable subsets of a cartesian product of $M$ of dimension 
$\leq 1$ define the same sets and have quantifier eliminations.
\end{enumerate}
\end{prop}

\begin{proof}
  In this setting, it is clear that a set is definable in $\CM$ over
  $\acl(\emptyset)$ if and only if it is definable in $\CM$ without
  parameters.

  By definition of dimension and the assumption that $\acl()=\dcl()$,
  a formula in $n$ variables that defines a one-dimensional set in $\CM$
  over $\emptyset$ is equivalent to an infinite disjunction
\begin{equation}\label{planarcurves}
\bigvee _{i=1} ^n \bigvee _{F\in \mathcal F} \mathbf{x} =F(x_i)
\end{equation}
where $\mathcal F$ is a set of $\emptyset$-definable $1$-variable functions from $M$ to $M^n$ and $\mathbf{x}=(x_1,\ldots , x_n)$.

By compactness, we can extract an equivalent finite disjunction from (\ref{planarcurves}), which gives the first part of the proposition.

The second part easily follows (either from a direct argument or from Theorem \ref{main}). 
 
\end{proof}

In Proposition \ref{acl}, we noted that in the case of a geometric structure $\CM$ on the universe $M$ with $\acl ()= \dcl ()$, the use of parameters 
in $\acl(\emptyset)$ is not needed in the statement of Theorem \ref{main}. These parameters are however essential for Theorem \ref{main} to hold in general: 
the theory of curves {\bf $\emptyset$-definable} in a geometric structure does not, in general, admit quantifier-elimination. 

\begin{ex}
\label{fatpoints}
Let $\CM =(M;+,\cdot,\mathfrak m )$ be a saturated algebraically closed valued field of characteristic $\neq 2$,
in the language of fields with a unary predicate  $\mathfrak m$ for its maximal ideal 
(the structure is known to be geometric; see for instance Section 4 in \cite{Mcp}). 

Denote by $i$ one of the two square roots of $-1$.

The formula $\rho$ in the free variable $(x,y)$  without parameters 
\[
\exists z \big( z^2+1=0 \wedge x-z\in \mathfrak m \wedge y-z\in \mathfrak m \big)
\]
is equivalent to 
\[
(x-i\in \mathfrak m \wedge y-i\in \mathfrak m ) \vee   (x+i\in \mathfrak m \wedge y+i\in \mathfrak m)
\]
which defines a boolean combination of $\acl(\emptyset)$-definable curve-based cylinders in $M^2$.

But $\rho$ does not define a boolean combination of $\emptyset$-definable curve-based cylinders.
\end{ex}

\begin{proof} 
Proceeding toward a contradiction suppose that  such a boolean combination exists. We can suppose that it is in a disjunctive normal form 
and each of the disjunctant is of the form 
\[
\phi_1(x)\wedge \phi_2(y)\wedge \phi_3 (x,y) \wedge \neg \phi_4 (x,y)
\]
where 
\begin{itemize}
 \item $\phi _1(x)$ and $\phi_2(y)$ are formul\ae \ without parameters in the language of valued field and 
 \item$\phi_3(x,y)$ and $\phi_4(x,y)$ each define a (possibly empty) subset of $M^2$ definable in $\CM$ without parameter which is either the whole $M^2$ 
or of dimension $\leq 1$.
\end{itemize}
  
Since the formula $\rho$ defines a set of dimension $2$, there is some
disjunctant such that $\phi _3$ is a tautology, the sets $\{x\in M|\CM
\models \phi _1(x)\}$ and $\{y\in M|\CM \models \phi _2(y)\}$ have
dimension one (precisely one; not zero or $-\infty$ !), and the
formula $\phi_4$ is not a tautology. Fix such a disjunctant.

Consider $\sigma \in \Aut (\CM / \emptyset)$ sending $i$ to $-i$. Let $\beta \in M\setminus \acl (\emptyset )$ such that 
$\phi _2(\beta)$ holds. We can find $\alpha$ such that \[
\phi_1(\alpha) \wedge \neg \phi_4(\alpha,\beta) \wedge \neg \phi_4(\alpha,\sigma (\beta))
\]
holds: since $\phi_4$ defines a set of dimension $\leq 1$ and the elements $\beta$ and $\sigma(\beta)$ are transcendental, the set 
$\{x\in M|\CM \models \phi_4(x,\beta) \vee \phi_4(x,\sigma (\beta))\}$ is finite and can not cover the infinite set $\{x\in M|\CM \models \phi _1(x)\}$.

Since all $\phi_i$'s are $\emptyset$-definable and $\phi_2(\beta)$ holds, so does 
$\phi_2(\sigma(\beta))$.

Therefore 
\[
\phi_1(\alpha) \wedge \phi_2(\beta) \wedge \neg \phi_4(\alpha,\beta)
\]
and 
\[
\phi_1(\alpha) \wedge \phi_2(\sigma(\beta)) \wedge \neg \phi_4(\alpha,\sigma(\beta))
\]
both hold. 

The three points $(\alpha,\beta)$, $(\sigma(\alpha),\sigma(\beta))$ and $(\alpha,\sigma(\beta))$ would satisfy $\rho$ which can not be. 
Indeed suppose for instance that $\alpha$ and $\beta$ belong to $i+\mathfrak m$. Then $\sigma(\beta) $ belongs both to $-i+\mathfrak m$ and 
to $\alpha+\mathfrak m=i+\mathfrak m$, which is impossible. (The case when $\alpha$ and $\beta$ belong to $-i+\mathfrak m$ is similar.)
\end{proof}

\begin{ex}
We get a similar result in $\CM =(\C;+,\cdot \ )$ by considering 
\[
\exists z \big( z^2+1=0 \wedge u+v=z \wedge w+t=z \big)
\]
in the free variables $(u,v,w,t)$.
\end{ex}

\section{Higher arities}
\label{high}
For each $n\in \N$, let $\mathcal C_n (\overline{\C}_{\C})$ 
(respectively $\mathcal C (\overline{\C}_{\C})$) denotes the structure on $\C$ 
whose basic definable sets are the  subsets of $\C^k$ for all $k\leq n$ (resp. for all $k\in \N$), 
$\C$-definable in $\overline \C$, of dimension $\leq 1$. 
Similarly, $\mathcal C_n (\overline{\C}_{\acl(\emptyset)})$ (respectively $\mathcal C (\overline{\C}_{\acl(\emptyset)})$) 
denotes the structure on $\C$ whose basic definable sets are the  
subsets of $\C^k$ for all $k\leq n$ (resp. for all $k\in \N$), $\acl(\emptyset)$-definable in $\overline \C$, of dimension $\leq 1$.
 
In Section \ref{sectionexemple}, we showed that the structure $\mathcal C_2(\overline \C_{\C})$ does not have quantifier elimination 
and therefore we obtained the existence of a constructible curve in $\C^3$ that is not equivalent to a boolean combination of cylinders 
whose basis are constructible curves in $\C^2$. Here we show:

\begin{prop} 
\label{nobound}
Given any natural number $n\geq 3$ there exists a $(n+1)$-ary relation $\emptyset$-definable in 
$\mathcal C_2(\overline \C_{\acl(\emptyset)})$ which is not quantifier-free 
definable in $\mathcal C_{n-1} (\overline{\C}_{\C})$.
\end{prop}

Proposition \ref{nobound} and Example \ref{exemple} give in particular that none of the structures $\mathcal C_n (\overline{\C}_{\C})$ have quantifier-elimination, 
for $n\geq 2$.

Since by Theorem \ref{main} any set definable in $\mathcal C_2(\overline \C_{\acl(\emptyset)})$ is equivalent to a boolean 
combination of cylinders whose basis are $\acl(\emptyset)$-definable curves of $\C^k$ for some $k\leq n+1$, we get:

\begin{cor}
\label{2(n-1)}
For any natural number $n\geq 2$ there is a subset of $\C^{n+1}$, $\acl(\emptyset)$-definable in $\overline \C$, of dimension $1$ which is not 
equivalent to any boolean combination of cylinders whose basis are $k$-curves with $k\leq \max \{2,n-1\}$, $\C$-definable in $\overline \C$.  
\end{cor}

\medskip

Let $\CM$ be a sufficiently  saturated extension of $\overline \C$ with universe $M$.

\begin{clm}
\label{AB}
There are two relations $S(s_1,\ldots,s_n,u)$ and
$T(t_1,\ldots,t_n,u)$ both $\emptyset$-definable in $\mathcal
C_2(\overline \C_{\acl (\emptyset )})$ such that we
have that
\begin{itemize}
\item[({\bf A})] if $a\in M\setminus \C$, $(s_1,\ldots,s_n,a)$ satisfies $S$ and $(t_1,\ldots,t_n,a)$
  satisfies $T$ then we have that
\[
(s_{\sigma(1)},\ldots,s_{\sigma(n-1)})\equiv _{\C\cup\{ a \}} (t_{\sigma(1)},\ldots,t_{\sigma(n-1)})
\] 
for all injection $\sigma$ from $\{1,\ldots ,n-1\}$ to $\{1,\ldots,n\}$ (the elementary equivalence being in the sense of $ACF_0$) . 
\item[({\bf B})] if $a\in M\setminus \C$, the sets 
\[
\{(u_1\ldots,u_n)\in M^n | \CM \models S(u_1,\ldots,u_n,a)\wedge \neg T(u_1,\ldots,u_n,a)\}
\]
and
\[
\{(u_1\ldots,u_n)\in M^n | \CM \models T(u_1,\ldots,u_n,a)\}
\]
are non-empty.
\end{itemize}
\end{clm}

The construction of such $S$ and $T$ and the proof that they satisfy these requirements will be the object of Lemmata \ref{elementaryequivalent},
\ref{nonempty} and \ref{binarization}. Let us admit for the moment their existence and prove Proposition \ref{nobound}.

\begin{proof}
Suppose for a contradiction that $T$ is equivalent to a boolean combination of formul\ae \  $\C$-definable in $\overline \C$, each involving at most $(n-1)$ of the 
possible variables. Let $U$ be one of these $(n-1)$-ary relations. Then there is some $k\leq n$ such that $U$ does not involve the $k^{\mathrm{th}}$ variable 
($U$ should also either not involve the last variable  or not involve the $l^{\mathrm{th}}$ variable for some $l\neq k \leq n$). 

Fix $a\in M\setminus \C$, $(\mathbf s,a)\models S$ and $(\mathbf t,a)\models T$. 

Since any subtuple of $(\mathbf s,a)$ of length $\leq n-1$ is elementary equivalent to the corresponding subtuple of  $(\mathbf t,a)$, since 
the relation $U$ involves at most $(n-1)$ variables and since $(\mathbf t,a)\models U$, we get that $(\mathbf s,a)\models U$.

The same being true for all such $U$, we get the implication 
\[
S(\mathbf s,a)\rightarrow T(\mathbf s,a),
\] 
a contradiction with ({\bf B}).
\end{proof}

Fix a natural number $N$. Given $a\in M$ we denote by $\Theta(a)$ the
set of roots of the polynomial $Z^N+Z^{N-1}+a$.  The following Lemma
tells us that the collection of sums of distinct elements of
$\Theta(a)$ is in bijection with the power set $\mathcal P (\Theta(a))$.  This will
allow us to encode some finite combinatorics in $\mathcal M$.

\begin{lem}
\label{elements}
Let $a\in M \setminus \C$. For each natural number $1\leq k \leq N$,
let $[ \Theta (a) ] ^k $ be the collection of all the subsets
of $\Theta (a)$ of size $k$.

The mapping from $[\Theta (a)]^k $ to $M$ sending $A$ to 
$\sum _{z \in A} z$ is injective.
\end{lem}

\begin{proof}

It follows from Galois theory (see the proof of Theorem 9 in \cite{Br}) that:
\[
\mathrm{Gal}\big( Z^N+Z^{N-1}+a\, / \,\C (a) \big)=\mathfrak S _N.
\]

Suppose for a contradiction that we have subsets $A\neq A'$ of
$\Theta(a)$ such that $|A|=|A'|$ and $\sum _{z \in A} z=\sum _{z
  \in A'} z$.  Without loss of generality we can assume that $|A|=|A'|$ minimal; in particular this implies  
$A\cap A'=\emptyset$.

If $|A|=1$, we clearly have a contradiction.
Thus we must have $|A|>1$. 

Assume first $\Theta(a)=A \cup A'$. Then $-1=\sum_{z \in A} z+\sum_{z
  \in A'} z$ so $\sum_{z\in A} z=-1/2\in \C$. Let $\zeta\in A$ and
$\zeta'\in A'$ be arbitrary chosen and let $\sigma$ be the element of
$\mathrm{Gal}\big( Z^N+Z^{N-1}+a\, / \,\C(a) \big)$ interchanging  $\zeta$
and $\zeta'$ and fixing the other roots. We have
\[ 
\sum_{z \in A} z = -1/2=\sigma(-1/2)=\sum_{z \in A} \sigma(z),
\]
hence  $\zeta=\zeta'$, contradicting the fact that $A\cap A'=\emptyset$.

We can thus assume $A\cup A'\neq \Theta(a)$. Let $\zeta\in A$,
$\zeta'\in \Theta(a) \setminus (A\cup A')$ and let 
$\sigma$ be the permutation interchanging $\zeta$ and $\zeta'$ and fixing the
other roots. We have
\[
\sum _{z \in A} \sigma(z)=
\sum _{\alpha \in A'} \sigma(z)=\sum _{z \in A'} z=\sum _{z \in A} z 
\]
that  gives $\zeta=\zeta'$, contradicting $\zeta'\notin A$.
\end{proof}

We now generalize the combinatorial configuration ``{\it triangle versus star}" appearing in the figure of Example \ref{exemple}. 

\begin{defi}
For $n\in \N$ we will denote by $\CL_n$ the first order language 
$\{P_1,\dotsc,P_n\}$ where each $P_i$ is a unary predicate ($\CL_0$ being the language of pure equality).

Let $n>1$ and $\mathcal F=\langle F; F_1,\dotsc F_n\rangle$ be an $\CL_n$-structure ({\it i.e.} 
$F_i$ is an interpretation of $P_i$ in $\mathcal F$).

We say that $\mathcal F$ 
is {\em symmetric} if for any permutation $\sigma \in \mathfrak S_n$ 
the structure $\mathcal F$ is isomorphic to the $\CL_n$-structure 
$\langle F; F_{\sigma(1)}\dotsc,F_\sigma(n)\rangle$ ({\it i.e.}
there is a bijection $\tilde \sigma\colon F\to F$ such that 
$\gamma\in F_i $ if and only if $\tilde\sigma(\gamma)\in F_{\sigma(i)}$).  
\end{defi}

\begin{lem}
\label{combi}
For any $n>1$ there are finite symmetric $\CL_n$-structures 
$\CX=\langle X; X_1,\dotsc, X_n\rangle$ and $\CY=\langle Y; Y_1,\dotsc, Y_n\rangle$
such that $\CX$ and $\CY$ are not isomorphic, but their reducts to $\CL_{n-1}$ are isomorphic.
\end{lem}

\begin{proof}

Set \[
X:=\{\alpha \subseteq \{1,\ldots ,n\}|\ |\alpha|\mbox{ is odd}\}
\]
and
\[
Y:=\{\beta \subseteq \{1,\ldots ,n\}|\ |\beta|\mbox{ is even}\}.
\]

For each $i\in \{1,\ldots ,n\}$, let 
\[
X_i:=\{\alpha \in X| i\in \alpha\}
\]
and
\[
Y_j:=\{\beta \in Y| j\in \beta \}.
\]

One can easily verify that $|X|=|Y|=2^{n-1}$ and that the $\CL_n$-structures $\CX=\langle X;X_1,\ldots,X_n\rangle$ and $\CY=\langle Y,Y_1,\ldots,Y_n\rangle$ 
are symmetric.

Consider the mapping
\[
\Phi: X \to Y
\] given by 
\[
\Phi(\alpha)= \begin{cases} \alpha \setminus \{n \}
				& \mbox{ if } n\in \alpha, \\
\alpha \cup \{ n \} & \mbox{ else.} \end{cases}.
\]
Clearly $\Phi$ is a bijection between $X$ and $Y$ and for $1\leq i \leq n-1$, we have 
\[
 (\alpha \in X_i) \Leftrightarrow (i\in \alpha ) \Leftrightarrow (i\in \Phi(\alpha)) \Leftrightarrow (\Phi(\alpha)\in Y_i).
\] 
That is, $\Phi$ is an isomorphism between the $\CL_{n-1}$-structures $\langle X;X_1,\ldots,X_{n-1}\rangle$ and 
$\langle Y;Y_1,\ldots,Y_{n-1}\rangle$.

Finally, to see that $\mathcal X$ and $\mathcal Y$  are not isomorphic
(as $\mathcal L_n$-structures), note that one and only one of the two sets $\bigcap_{1\leq i \leq n} X_i$ and
$\bigcap_{1\leq j \leq n} Y_j$ is non-empty: 
\begin{itemize}
\item if $n$ is even then $\bigcap_{1\leq i \leq n} X_i=\emptyset$ and
  $\bigcap_{1\leq i \leq n} Y_i=\{1,\ldots , n\}$, and
\item if $n$ is odd then $\bigcap_{1\leq i \leq n} Y_i=\emptyset$ and
  $\bigcap_{1\leq i \leq n} X_i=\{1\ldots,n\}$.
\end{itemize} 
Therefore there is no bijection between $X$ and $Y$ sending each $X_i$
to $Y_i$.
\end{proof}

For the rest of the Section, we let $\CX=\langle
X;X_1,\ldots,X_n\rangle$ and $\CY=\langle Y;Y_1,\ldots,Y_n\rangle $ be
two symmetric $\CL_n$-structures satisfying the conclusion of Lemma
\ref{combi}. We let $N=|X|=|Y|$ and, as in Lemma \ref{elements}, we let
$\Theta(a)$ denote the set of roots of the polynomial $Z^N+Z^{N-1}+a$. 
Note that if $a\in M\setminus \C$ then $Z^N+Z^{N-1}+a$ has $N$ distinct roots, and
$|\Theta(a)|=|X|=|Y|$. 

\medskip
Consider now the relations $S'$ and $T'$ given by 
\begin{enumerate}
\item $S'(s_1,\ldots,s_n,a)$ holds 
if and only if  
there is a bijection $\phi$ between $X$ and $\Theta(a)$ such that 
\[
\mbox{for all $1\leq i \leq n$, }  s_i=\sum_{\alpha \in X_i} \phi(\alpha)
\]
\item $T'(t_1,\ldots,t_n,a)$ holds if and only if  
there is a bijection $\psi$ between $Y$ and $\Theta(a)$ such that 
\[
\mbox{for all $1\leq i \leq n$, }  t_i=\sum_{\beta \in Y_i} \psi(\beta).
\]
\end{enumerate}
Using Lemma \ref{elements} to transfer the combinatorial properties of $\CX$ and $\CY$, we will show that these relations, definable in $\CM$, 
satisfy properties ({\bf A}) and ({\bf B}) of Claim \ref{AB}.

\begin{lem}
\label{elementaryequivalent}
Fix $a\in M\setminus \C$. 
Let   $\phi$ be a  bijection between $X$ and $\Theta(a)$,
and  $\psi$ be a bijection between $Y$ and $\Theta(a)$.

For $i\in\{1,\ldots,n\}$, let 
\[
s_i=\sum_{\alpha \in X_i} \phi(\alpha)
\]
and
\[
t_i=\sum_{\beta \in Y_i} \psi(\beta).
\]

Then the tuples 
$(s_{\sigma(1)},\ldots,s_{\sigma(n-1)})$ and $(t_{\tau(1)},\ldots,t_{\tau(n-1)})$ are elementary equivalent  over $\C\cup\{a\}$ (in  the theory of  $\overline \C$)
for all injections $\sigma$ and $\tau$ from $\{1,\ldots,n-1\}$ to $\{1,\ldots,n\}$.
\end{lem}

\begin{proof}
By the choice  of $\CX$ and $\CY$, there is a bijection $\lambda$ between
 $X$ and $Y$ that sends each set $X_{\sigma(i)}$ to the corresponding set $Y_{\tau(i)}$ for $i=1,\ldots,n-1$.

But as noted in Lemma \ref{elements}, the Galois group of $Z^N+Z^{N-1}+a$ over $\C(a)$ is $\mathfrak S _N$.
Therefore the bijection $\phi (\alpha)\mapsto \psi (\lambda(\alpha))$ of $\Theta(a)$ extends to a $\CM$-automorphism $\Lambda$ of $M$ fixing  $\C \cup \{a \}$.
We now have
\[
t_{\tau(i)}=\sum_{\beta \in Y_{\tau(i)}} \psi(\beta)=\sum_{\alpha \in X_{\sigma(i)}} \psi(\lambda(\alpha))=
\Lambda\big( \sum_{\alpha \in X_{\sigma(i)}} \phi(\alpha)\big) =\Lambda(s_{\sigma(i)}):
\]

$\Lambda$ sends $(s_{\sigma(1)},\ldots,s_{\sigma(n-1)})$ to $(t_{\tau(1)},\ldots,t_{\tau(n-1)})$; in particular these two tuples are elementary equivalent 
over $\C\cup\{a\}$ modulo the theory of $\overline \C$.

\end{proof}

\begin{lem}
\label{nonempty}
Let $a\in M\setminus \C$, the sets 
\[
\{(u_1\ldots,u_n)\in M^n | \CM \models S'(u_1,\ldots,u_n,a)\wedge \neg T'(u_1,\ldots,u_n,a)\}
\]
and
\[
\{(u_1\ldots,u_n)\in M^n | \CM \models T'(u_1,\ldots,u_n,a)\}
\]
are non-empty.
\end{lem}

\begin{proof}

Fix $a\in M\setminus \C$.

By definition of $T'$, it is clear that there is some $(u_1\ldots,u_n)\in M^n$ for which $T'(u_1,\ldots,u_n,a)$ holds. 

Similarly, we can find some $(s_1,\ldots,s_n)\in M^n$ 
such that $S'(s_1,\ldots,s_n,a)$ holds. Suppose $T'(s_1,\ldots,s_n,a)$ also holds.
Then we get some bijections $\phi:X\to \Theta(a)$ and $\psi:Y\to \Theta(a)$
\[
 \mbox{for all $1\leq i \leq n$, } \sum_{\alpha \in X_i} \phi(\alpha)=\sum_{\beta \in Y_i} \psi(\beta).
\]

By Lemma \ref{elements} we thus get that $\phi(X_i)=\psi(Y_i)$ for all $1\leq i \leq n$: $\psi^{-1}\circ \phi$ would be an isomorphism between 
$\CX$ and $\CY$. This is can not be, hence $T'(s_1,\ldots,s_n,a)$ must fail.
\end{proof}

It now remains to replace the formul\ae \ $S'$ and $T'$ by formul\ae \ definable in $\mathcal B(\overline \C_{\emptyset})$:

\begin{lem}
\label{binarization}
There are relations $S(s_1,\ldots,s_n,u)$ and $T(t_1,\ldots,t_n,u)$
definable in $\mathcal B(\overline \C_{\emptyset})$ such that for all 
$a\in M\setminus \C$ we have that
\begin{enumerate}
\item $S(s_1,\ldots,s_n,a)$ holds if and only  $S'(s_1,\ldots,s_n,a)$ holds and
\item $T(s_1,\ldots,s_n,a)$ holds if and only  $T'(s_1,\ldots,s_n,a)$ holds
\end{enumerate}

\end{lem}

\begin{proof}
   Consider $R(u,v)$ the binary relation that says that $u$ is the sum
  of $N'=|X_1|$ distinct roots of the polynomial
  $Z^N+Z^{N-1}-(v^N+v^{N-1})$, one of which is $v$. That is $R(u,v)$
  holds if and only if
\begin{multline*}
  \exists (z_1, \ldots ,z_{N'-1},z) \ \big( u=v+ \sum_{k=1} ^{N'-1} z_k \wedge v\in \Theta (z) 
\wedge \bigwedge_{k=1} ^{N'-1} z_k\in \Theta(z)  \\
  \wedge \bigwedge_{k=1} ^{N'-1} v\neq z_k \wedge \bigwedge_{k\neq l}
  z_k \neq z_l \big).
\end{multline*}
The relation $R$ is definable in $\overline \C$ without parameters.

Let $\mathbf{x}$ be a $N$-tuple of variables $(x_{\alpha})_{\alpha\in X}$ indexed by $X$ and consider the $(n+1)$-ary relation $S$ defined by 
\begin{multline*}
S(s_1,\ldots,s_n,u) \Leftrightarrow \\
\big( \exists \mathbf{x} 
\bigwedge_{\alpha\in X} x_{\alpha}\in\Theta(u) 
\wedge \bigwedge_{\substack{\alpha ,\alpha ' \in X \\ \alpha\neq \alpha '}} x_{\alpha}\neq x_{\alpha '} 
\wedge \bigwedge _{\alpha \in X_i} R(s_i, x_{\alpha})
\big).
\end{multline*}

Let $\mathbf{y}$ be a $N$-tuple of variables $(y_{\beta})_{\beta\in Y}$ indexed by $Y$ and consider the $(n+1)$-ary relation $T$ defined by 
\begin{multline*}
T(t_1,\ldots,t_n,u) \Leftrightarrow \\
\big( \exists \mathbf{y} 
\bigwedge_{\beta\in Y} y_{\beta}\in\Theta(u) 
\wedge \bigwedge_{\substack{\beta ,\beta ' \in Y^2 \\ \beta\neq \beta '}} y_{\beta}\neq y_{\beta '} 
\wedge \bigwedge _{\beta \in Y_j} R(t_j, y_{\beta})
\big).
\end{multline*}

These relations are definable in $\mathcal B(\overline \C_{\emptyset})$ (the relations $R$ and ``$z\in \Theta (u)$'' being binary). 
We will show that they fulfill the conditions (1) and (2) of the Lemma. 

Fix  $a\in M\setminus \C$.

It is clear that if $S'(s_1,\ldots,s_n,a)$ holds then  $S(s_1,\ldots,s_n,a)$ holds. Reciprocally, let $(s_1,\dots,s_n)\in M^n$ be such that
$S(s_1,\ldots,s_n,a)$ holds.  By definition of $S$, we can find a
bijection $\phi:X\to \Theta(a)$ and, for each $i$ and each $\alpha\ni
i$, an injection $\phi _{i,\alpha}:X_i \to \Theta(a)$, such that for
each $1\leq i\leq n$,
\[
s_i=\sum _{\alpha '\in X_i} \phi_{i,\alpha} (\alpha ')\mbox{ and } \phi_{i,\alpha}(\alpha)=\phi(\alpha).
\]

Fix such an $i$. Consider $\alpha$ and $\alpha'$ in $X_i$. By Lemma
\ref{elements}, $\phi_{i,\alpha}$ and $\phi_{i,\alpha'}$ have the same
range and therefore $\phi(\alpha')$ belongs to the range of
$\phi_{i,\alpha}$ for all $\alpha'\in X_i$. We thus have that
$\phi(X_i)=\phi_{i,\alpha}(X_i)$ for some (all) $\alpha \in X_i$
and \[s_i=\sum_{\alpha \in X_i} \phi(\alpha).	\]

The proof of (2) is similar. 
\end{proof}

Putting Lemmata \ref{binarization}, \ref{elementaryequivalent} and \ref{nonempty} together, we get, as announced in Claim \ref{AB}, two relations $S$ and $T$ 
definable in  $\mathcal B(\overline \C_{\emptyset})$ that satisfy the conditions ({\bf A}) and ({\bf B}).

\section{Definability}
\label{quantifier}

Since example \ref{exemple} and Proposition \ref{nobound} show that for
any fixed $n\geq 2$, there are sets definable in $\mathcal C (\overline \C
_{\C})$ (the structure on $\C$ whose basic relations are all the algebraic curves, of any arity) which are not \emph{quantifier-free} definable in $\mathcal
C_n (\overline \C _{\C})$ (the structure on $\C$ whose basic relations are all the algebraic curves of $\C ^n$), it is natural to ask if all the sets
definable in $\mathcal C (\overline \C _{\C})$ are definable in some $\mathcal C_n (\overline \C _{\C})$ (allowing, this time, quantifiers).

\begin{prop}
The two structures $\mathcal C_3 (\overline \C _{\C})$  and 
$\mathcal C (\overline \C _{\C})$ define the same sets.
\end{prop}

\begin{proof}
 By quantifier-elimination for $\mathcal C (\overline \C _{\C})$, it suffices to show that any algebraic curve is definable in $\mathcal C_3 (\overline \C _{\C})$.
But it is well known that any affine curve $Y\subset \C^n$ is bi-rational to a planar curve $X\subset \C^2$ 
(see for example Chapter I, Section 3.3, Theorem 5 of \cite{Sh}).
Let $\phi=(\phi_1,\cdots,\phi_n)$ be such an isomorphism. Each restriction of $\phi_i$ to $X$ is a basic definable set in $\mathcal C_3 (\overline \C _{\C})$ 
thus the graph $\Gamma$ of the restriction of $\phi$ to $X$ is (quantifier-freely) definable in $\mathcal C_3 (\overline \C _{\C})$ and $Y$, which 
is the union of the projection of $\Gamma$ on the last $n$ coordinates and finitely many points, is definable in $\mathcal C_3 (\overline \C _{\C})$.
\end{proof}

\begin{rem}
 From the proof, we see that the depth of alternation of quantifier for fomul\ae\ in the language with a symbol for each algebraic curve of $\C ^3$ is at most $1$.
The lack of quantifier-elimination implies that this maximal depth is realized. 
\end{rem}

\begin{question}
 Is $\mathcal C_2 (\overline \C _{\C})$ a proper reduct (in the sense of definability) of $\mathcal C (\overline \C _{\C})$  ? 
\end{question}

\end{document}